\documentclass{amsart}

\usepackage{amsmath}
\usepackage{amssymb}
\usepackage{amsthm}

\newtheorem{theorem}{Theorem}[section]
\newtheorem{lemma}[theorem]{Lemma}

\theoremstyle{definition}

\theoremstyle{remark}
\newtheorem{remark}[theorem]{Remark}

\begin{document}


\title{Density of smooth functions in variable exponent Sobolev spaces}

\author{Thanasis Kostopoulos}
\email{thanoskostopoulos80@gmail.com}

\author{Nikos Yannakakis}
\email{nyian@math.ntua.gr}


\address{Department of Mathematics, National Technical University of Athens,
Zografou campus, Athens 15780, Greece}

\begin{abstract}
We show that if $p_-\geq 2$, then a sufficient condition for the density of smooth functions with compact support, in the variable exponent Sobolev space $W^{1,p(\cdot)}(\mathbb R^n)$, is that the Riesz potentials of compactly supported functions of $L^{p(\cdot)}(\mathbb R^n)$, are also elements of $L^{p(\cdot)}(\mathbb R^n)$.
Using this result we then prove that the above density holds if (i) $p_-\geq n$ or if (ii) $2\leq p_-< n$ and $p_+<\frac{np_-}{n-p_-}$.
Moreover our result allows us to give an alternative proof, for the case $p_-\geq 2$, that the local boundedness of the maximal operator and hence local log-H{\"o}lder continuity imply the density of smooth functions with compact support, in the variable exponent Sobolev space $W^{1,p(\cdot)}(\mathbb R^n)$.
\end{abstract}


\maketitle

\section{Introduction}
Let $\Omega$ be an open subset of $\mathbb R^{n}$. In what follows
$p\colon\Omega\to(1,\infty)$ is a measurable function, such that
$$1<p_-=\mathrm{ess}\inf_{x\in\Omega}\, p(x)\leq p_+ = \mathrm{ess}\sup_{x\in\Omega}\ \, p(x)<+\infty\,.$$

The variable exponent Lebesgue space $L^{p(\cdot)}(\Omega)$ is defined as the function space containing all measurable functions $f$ on $\Omega$, such that
\begin{equation}
\nonumber
\varrho_{p(\cdot)}(f)=\int_{\Omega} |f(x)|^{p(x)}\,dx<\infty\,.
\end{equation}

Equipped with the norm
\begin{equation}
\nonumber
\|f\|_{p(\cdot)} = \inf \big\{ \lambda > 0 : \varrho_{p(\cdot)}(f/\lambda) \leq 1 \big\}\,,
\end{equation}
the space $L^{p(\cdot)}(\Omega)$ becomes a separable and reflexive Banach space.

The variable exponent Sobolev space $W^{1,p(\cdot)} (\Omega)$ is defined as
\begin{equation}
\nonumber
W^{1,p(\cdot)} (\Omega)=\left\{f\in L^{p(\cdot)} (\Omega): D_if\in L^{p(\cdot)} (\Omega)\,,1\leq i\leq n\right\}\,,
\end{equation}
where by $D_i f$ we denote the $i$-th distributional partial derivative of $f$.

The space $W^{1,p(\cdot)}(\Omega)$ equipped with the norm
\begin{equation}
\nonumber
\| f\|_{1, p(\cdot)}=\| f\|_{p(\cdot)}+\sum_{i=1}^n\| D_i f\|_{p(\cdot)}
\end{equation}
is also a separable and reflexive Banach space.

Recently variable exponent Lebesgue and Sobolev spaces have attracted a lot of attention, mainly due to their use in PDE's and variational problems with non-standard growth (two celebrated applications being the modelling of electrorheological fluid and image processing).

An important difference between classical and variable exponent Sobolev spaces was pointed out quite early by V. Zhikov. In particular in \cite{zhikov2} he showed that non-regularity of the Lagrangian involved in a certain variational problem has a striking implication: smooth functions are not dense in the corresponding variable exponent Sobolev space. Note that variable exponent Lebesgue spaces are not translation invariant and hence convolution, being the main tool in the classical case, cannot be taken for granted.

Since then many researchers have tried to overcome this problem. First D. Edmunds and J. Rakosnik in \cite{edmunds} used a local monotonicity condition that allowed smoothing by convolution with an appropriately chosen anisotropic kernel.

A decisive step towards the understanding of the problem was made by the introduction of the so-called log-H{\"o}lder continuity.
Recall, see \cite[Definition 2.2]{cruz1}, that a function $p:\Omega\rightarrow\mathbb R$ is called
locally log-H{\"o}lder continuous on $\Omega$, if there exists $c_0>0$, such that
\begin{equation}
\nonumber
|p(x)-p(y)| \leq \frac{c_0}{-\log (\big|x-y| \big)}\,,
\end{equation}
for all $x,y\; \in \Omega$, with $|x-y|<1/2$.

S. Samko in \cite{samko1}, L. Diening in \cite{dhn} and D. Cruz-Uribe with A. Fiorenza in \cite{cruz} proved independently that local log-H{\"o}lder continuity is sufficient for the density of smooth functions, their success depending on that under this assumption some of the useful properties of convolution ``survive''. Note that a deeper reason why local log-H{\"o}lder continuity is sufficient is that it implies the local boundedness of the maximal operator, which in its turn guarantees density see \cite[Theorems 6.14 and 6.17]{cruz1}.

V. Zhikov returned to the problem in \cite{zhikov1} and proved that a weaker, than log-H{\"o}lder, modulus of continuity will suffice. We should mention here that V. Zhikov avoided the problematic convolution by using the fact that approximation by Lipschitz functions in the corresponding variable exponent Sobolev space is sufficient for the density of smooth ones.

Finally, P. Hasto in \cite{hasto} and X. Fan et al. in \cite{fan} proved independently that a combination of the monotonicity condition of D. Edmunds and J. Rakosnik in \cite{edmunds} with local log-H{\"o}lder continuity implies the required density.

Our main result in this paper is that if $p_-\geq 2$, then a sufficient condition for the density of $C^\infty_0(\mathbb{R}^n)$ in $W^{1,p(\cdot)} (\mathbb{R}^n)$ is
\begin{equation}
\label{hypo}
I_1(f)\in L^{p(\cdot)}(\mathbb{R}^n)\,,
\end{equation}
for any compactly supported function $f\in L^{p(\cdot)}(\mathbb{R}^n)$, where $I_1(f)$ is the Riesz potential of $f$. Recall that the Riesz potential $I_1(f)$ of a function  $f$ is defined by
\begin{equation}
\nonumber
I_1 (f)(x)=c\int_{\mathbb R^{n}} \frac{f(x-y)}{|x-y|^{n-1}}\,dy\,,
\end{equation}
with $c$ being a positive constant, depending on $n$, which will be defined in the sequel.

Note that if $p(\cdot)$ is constant, then the boundedness of the maximal operator in $L^p(\mathbb R^n)$, implies that (\ref{hypo}) is always true for compactly supported functions, see \cite[Proposition 6.22 and Remark 6.23]{cruz1}.

Applying the above we are able to present two density results:
the first one guarantees that the following conditions 
\begin{itemize}
\item[(i)] $p_-\geq n\,$ or
\item[(ii)] $2\leq p_-< n$ and $p_+ \leq \frac{n p_-}{n- p_-}$
\end{itemize}
are sufficient for density. We should mention here that (i) and (ii) are frequently used in the context of variable exponent Sobolev spaces, see for example \cite[Lemma 8.2.14]{dhhr}.

The second one states that if $p_-\geq 2$,  then the local boundedness of the maximal operator implies density and consequently, as we mentioned above, so does local log-H{\"o}lder continuity. Hence, under the restriction $p_-\geq 2$, we have an alternative proof of this important result.

Concluding, it seems interesting that two apparently unrelated types of results, the first concerning the relation of the exponent to the dimension of the space and the second concerning its regularity, are both justified by the same source: the ``nice''  behavior of the Riesz potential of functions with compact support. It is our hope that this fact may shed some additional light to the understanding of the subtle problem of density of smooth functions in variable exponent Sobolev spaces.


\section{Preliminaries}
In this section we present some results that we will need in the sequel. For more details concerning variable exponent function spaces, we refer the interested reader to the books \cite{cruz1}  and \cite{dhhr}.

We begin by a result that can be found in \cite[Theorem 3.3.11]{dhhr}. In what follows the notation $X\hookrightarrow Y$ means that $X$ is continuously embedded into $Y$, i.e. $X\subseteq Y$ and there exists $c>0$ such that
$$\|x\|_Y\leq c\|x\|_X\,,\text{ for all }x\in X\,.$$
\begin{theorem}
\label{variable}
Let $p, q, r$ be measurable functions on $\Omega$, with $p \leq q \leq r\leq \infty$ a.e. in $\Omega$. Then
\begin{equation}
\nonumber
L^{p(\cdot)} (\Omega) \cap L^{r(\cdot)} (\Omega) \hookrightarrow L^{q(\cdot)} (\Omega)\,.
\end{equation}
\end{theorem}

The next result is useful, since it allows us to work with bounded and compactly supported  functions.
\begin{theorem}\cite[Corollary 9.1.4]{dhhr}
\label{bounded}
Bounded Sobolev functions with compact support are dense in $$W^{1,p(\cdot)}(\mathbb R^{n})\,.$$
\end{theorem}

We will also use the following integral representation of $W^{1,1}(\mathbb{R}^n)$ functions. In what follows by $\omega_n$ we denote the volume of the unit ball of $\mathbb{R}^n$.
\begin{lemma}\cite[Lemma 7.14]{gilbarg}
\label{representation}
Let $f\in W^{1,1}(\mathbb{R}^n)$. Then
\begin{equation}
\nonumber
f(x)=\frac{1}{n\omega_n}\int_{\mathbb{R}^n}\sum_{i=1}^n\frac{(x_i-y_i)D_if(y)}{|x-y|^n}\,dy\,,\text{ a.e. in }\mathbb{R}^n\,.
\end{equation}
\end{lemma}

Finally, if $0 < \alpha < n$, the Riesz potential of a locally integrable function $f$ is defined by
\begin{equation}
\nonumber
I_\alpha (f)(x)=\frac{1}{\gamma (\alpha)}\frac{1}{|x|^{n-\alpha}}*f=\frac{1}{\gamma (\alpha)} \int_{\mathbb R^{n}} \frac{f(x-y)}{|x-y|^{n- \alpha}}\,dy\,,
\end{equation}
where
$$\gamma (\alpha) = \pi ^{n/2} 2^{\alpha} \frac{\Gamma (\frac{\alpha}{2})}{\Gamma (\frac{n- \alpha}{2})}\,.$$

It is known, see for example \cite[Theorem 1, p. 119]{stein}, that if $f\in L^p(\mathbb{R}^n)$, then
\begin{equation}
\label{riesz1}
I_\alpha(f) \in L^{\frac{np}{n- \alpha p}}(\mathbb{R}^n)\,.
\end{equation}

It is important to notice, as we have already mentioned in the Introduction, that variable exponent functions spaces are not translation invariant. In fact, translations are bounded on $L^{p(\cdot)} (\Omega)$, see \cite[Proposition 3.6.1]{dhhr}, if and only if the variable exponent $p(\cdot)$ is constant. An immediate consequence of this, is the ``failure of convolution'', meaning that Young's inequality does not in general hold.
\section{Main result}
It is well-known, see for example \cite[p.23]{Riesz}, that if $f$ is a rapidly decreasing continuous function then
$$\lim_{\alpha \rightarrow 0^+}I_\alpha(f)(x)=f(x)\,,\text{ for all }x\in\mathbb R^n\,.$$
In the next lemma we show that a similar thing holds also for $L^p(\mathbb R^{n})$ functions, with compact support, as long as $p\geq 2$.
\begin{lemma}
\label{lemma1}
Let $p\geq 2$. If $f\in L^{p}(\mathbb R^{n})$ and has compact support, then
\begin{equation}
\nonumber
\lim_{\alpha \rightarrow 0^+} I_\alpha(f)=f, \text{ in } L^2(\mathbb{R}^n).
\end{equation}
\end{lemma}
\begin{proof}
Since $p\geq 2$ and $f$ has compact support we have that
$$f\in L^{\frac{2n}{n+2\alpha}}(\mathbb{R}^n)\,,$$  for all $0<\alpha<n$. Let
$$c_{\alpha} = \pi^{-\alpha/2} \Gamma (\frac{\alpha}{2})$$
and
$$g_\alpha=c_{n-\alpha}\frac{1}{|x|^{n-\alpha}}*f\,.$$
By (\ref{riesz1}) the functions $g_\alpha$ belong to $L^2(\mathbb{R}^n)$. Thus if we additionally assume that $\alpha<n/2$, then using  \cite[Corollary 5.10]{ll} we have that
\begin{equation}
\label{fourier}
c_{\alpha} |k|^{-\alpha} \widehat{f}(k) = \widehat{g}_\alpha(k),
\end{equation}
where  $\widehat{f}$, $\widehat{g}_\alpha$ are the Fourier transforms of $f$ and $g_\alpha$ respectively.

Since $\| \widehat{g}_\alpha \|_{2} = \| g_\alpha \|_{2}$ we get that,
\begin{equation}
\nonumber
\int_{\mathbb{R}^n} |c_{\alpha} |k|^{-\alpha} \widehat{f}(k)|^{2} dk = \int_{\mathbb{R}^n} |g_\alpha(k)|^{2} dk\,.
\end{equation}
Hence
\begin{equation}
\label{riesz}
(2\pi)^{-2\alpha} \int_{\mathbb{R}^n} \frac{|\widehat{f}(k)|^{2}}{|k|^{2\alpha}} dk = \int_{\mathbb{R}^n} | (I_{\alpha}f)(k) |^{2} dk\,.
\end{equation}
If $\alpha<1/2$ and taking in mind that $f$ is also in $L^1(\mathbb{R}^n)$ (so its Fourier transform is in $L^{\infty}(\mathbb{R}^n)$), we have that
\begin{equation}
\nonumber
\frac{|\widehat{f}(k)|^{2}}{|k|^{2\alpha}}\leq \chi _{B(0,1)}(k) \frac{\| \widehat{f} \|_{\infty}^{2}}{|k|} + \chi _{B(0,1)^{c}}(k)|\widehat{f}(k)|^{2}\,.
\end{equation}

Since the right-hand side of the above inequality is an $L^1$-function, taking limits in (\ref{riesz}) and using the dominated convergence theorem we have that
\begin{equation}
\label{norms}
\lim _{\alpha \rightarrow 0^+} \int_{\mathbb{R}^n} | (I_{\alpha}f)(k) |^{2} dk = \int_{\mathbb{R}^n} |\widehat{f}(k)|^{2} dk = \int_{\mathbb{R}^n} |f(k)|^{2} dk\,.
\end{equation}
On the other hand from (\ref{fourier}) we get that
\begin{equation}
\nonumber
\widehat{I_{\alpha} (f)}(k)= \frac{2^{- \alpha}}{\pi^{\alpha}} \frac{\widehat{f}(k)}{|k|^{\alpha}}.
\end{equation}
Multiplying both sides by $\bar{\widehat f}$ and integrating, we have
\begin{equation}
\nonumber
\int_{\mathbb{R}^n} \widehat{I_{\alpha} (f)(k)}\overline{\widehat{f}(k)}\, dk = \frac{2^{- \alpha}}{\pi^{\alpha}} \int_{\mathbb{R}^n} \frac{|\widehat{f}(k)|^2}{|k|^{\alpha}}\, dk\,.
\end{equation}
Then using Parseval's formula we get that
\begin{equation}
\nonumber
\int_{\mathbb{R}^n} I_{\alpha} (f)(k) \overline{f(k)}\, dk = \frac{2^{- \alpha}}{\pi^{\alpha}} \int_{\mathbb{R}^n} \frac{|\widehat{f}(k)|^2}{|k|^{\alpha}}\, dk.
\end{equation}
As before, it is easy to see that $\frac{ | \widehat{f}(k) |^{2} }{|k|^{\alpha}} $ is bounded by an $L^{1}$-function and thus taking limits and using the dominated convergence theorem again we have that
\begin{equation}
\label{weak}
\lim _{\alpha \rightarrow 0^+} \int_{\mathbb{R}^n} I_{\alpha} (f)(k) \overline{f(k)} dk = \int_{\mathbb{R}^n} |\widehat{f}(k)|^{2} dk = \int_{\mathbb{R}^n} |f(k)|^{2} dk.
\end{equation}
Hence by (\ref{norms}) and (\ref{weak}) we have that $\| I_{\alpha}(f)\|_{L^2}\rightarrow\|f\|_{L^2}$, $\langle I_{\alpha}(f), f\rangle\rightarrow \|f\|_{L^2}$, where by $\langle\cdot,\cdot\rangle$ we denote the inner product of $L^{2}(\mathbb{R}^n)$. So
$I_{\alpha}(f) \rightarrow f$ in $L^{2}(\mathbb{R}^n)$, as $\alpha\rightarrow 0^+$.
\end{proof}

The next lemma provides us with a useful tool in order to approximate a Sobolev function with compactly supported smooth ones.

\begin{lemma}
\label{lemma2}
There exists a sequence $(g_\lambda)$, $\lambda>1$, of functions in $C_0^\infty(\mathbb R^n)$, which are $0$ in the ball $B(0,\frac{1}{\lambda} )$, $1$ in
$B(0, \lambda) \setminus B(0, \frac{2}{\lambda})$ and $0$ in $\mathbb R^n\setminus B(0, 2 \lambda)$, such that their first derivatives are bounded by the function $\frac{c_1}{|x|}$, where $c_1$ is a positive constant.
\end{lemma}
\begin{proof}
We consider the functions
$$f(x)= \left\{\begin{array}{ccl} e^{-\frac{1}{x-1}}&,&\text{ if }x>1 \\ 0&,& \text{ if }x\leq1
\end{array}\right.$$
and
\begin{equation}
\nonumber
\qquad \qquad \psi(x) = \frac{f(x)}{f(x)+f(3-x)}.
\end{equation}
The denominator is always different from zero, the first derivative of $f$ is continuous and bounded, so  the first  derivative of $\psi$ is continuous and bounded.
Then defining $\psi_\lambda$ on $\mathbb R^n$, with
$$\psi_\lambda(x)=\psi(\lambda |x|)$$
we have that
\begin{eqnarray*}
0<\psi_\lambda (x)<1&,& \text{ if }\frac{1}{\lambda} < |x| < \frac{2}{\lambda}\\
\psi_\lambda (x)=0&,& \text{ if } |x| \leq \frac{1}{\lambda}\\
\psi_\lambda (x)=1&,& \text{ if }\ |x| \geq \frac{2}{\lambda}.
\end{eqnarray*}
It is easy to see that $\psi_\lambda \in C^\infty(\mathbb R^n)$ and its first derivatives are bounded by the function $\frac{c_1}{|x|}$,  for some positive constant $c_1$.
Analogously, we construct the functions $\psi_\lambda'(x) \in C^\infty(\mathbb R^n)$, with
$$\psi_\lambda'(x) = \psi (\frac{|x|}{\lambda})$$
and
\begin{eqnarray*}
0<\psi_\lambda' (x)<1&,& \text{ if }\lambda < |x| < 2 \lambda\\
\psi_\lambda' (x)=0&,& \text{ if }|x| \leq \lambda\\
\psi_\lambda' (x)=1&,& \text{ if }|x| \geq 2 \lambda,
\end{eqnarray*}
which have their derivatives bounded by a constant $c$.

It can be easily seen that the functions
$$g_\lambda = \psi_\lambda(1- \psi_\lambda')$$
have the desired properties and hence the proof is completed.
\end{proof}

Our main result is the following.

\begin{theorem}
\label{theorem1}
Let $p_-\geq 2$ and assume that $$I_1(f)\in L^{p(\cdot)}(\mathbb{R}^n)\,,$$ for any $f\in L^{p(\cdot)}(\mathbb{R}^n)$, with compact support. Then
$C_0^\infty (\mathbb R^n)$ is dense in $W^{1,p(\cdot)} (\mathbb R^{n})$.
\end{theorem}

\begin{proof}
Let $f$ be a function in $W^{1,p(\cdot)} (\mathbb R^{n})$. By Theorem \ref{bounded}, without loss of generality,  we may assume that $f$ is bounded and has compact support. Fix $1\leq i\leq n$ and consider the sequence of functions
\begin{equation}
\nonumber
\omega_\lambda^i(x) = \int_{\mathbb R^n} g_\lambda (y) \frac{y_{i}}{|y|^{n-\frac{1}{\lambda^2}}} D_if(x-y) dy,
\end{equation}
where $g_\lambda$ are as in Lemma \ref{lemma2} and $\lambda\in\mathbb N$. The functions $g_\lambda (y) y_{i}|y|^{-n+\frac{1}{\lambda^2}}$ are in $C_{0}^{\infty}(\mathbb R^{n})$ and $D_if \in L^{1}$ has compact support. Thus $\omega_\lambda^i \in {C_{0}^\infty (\mathbb R^n)}$.
We will show that the $|\omega_\lambda^i|$'s are bounded above by an $L^{p(\cdot)}$-function. We have that
\begin{eqnarray*}
| \omega_\lambda^i (x) |&=&  \big| \int_{\mathbb R^n} g_\lambda (y) \frac{y_{i}}{|y|^{n-\frac{1}{\lambda^2}}} D_if(x-y) dy \big|\\
&\leq &  \int_{\mathbb R^n} \frac{g_\lambda(y)|y|^{\frac{1}{\lambda^2}}}{|y|^{n-1}} \big| D_if(x-y) dy \big|\\
&\leq &  \int_{B(0,2\lambda)} \frac{|y|^{\frac{1}{\lambda^2}}}{|y|^{n-1}} \big| D_if(x-y) dy \big|.
\end{eqnarray*}
where we have used Lemma \ref{lemma2} in the last step. Using the fact that $(2 \lambda)^{\frac{1}{\lambda^2}} < 2$, for every $\lambda > 1$, we have that
\begin{equation}
\nonumber
| \omega_\lambda^i (x) |  \leq \,c I_1(|D_i(f)|)\,.
\end{equation}
But $D_i(f)$ has compact support and hence by hypothesis $I_1(|D_i(f)|)$ belongs to $L^{p(\cdot)}(\mathbb R^{n})$.

Next we will prove that a subsequence of $ \{ \omega_\lambda^i \}$,
converges to the function
\begin{equation}
\nonumber
f_i (x) = \int_{\mathbb R^n} \frac{y_i}{|y|^n} D_if(x-y)dy\, , \text{ a.e in }\; \mathbb R^n\,.
\end{equation}

We have that
\begin{equation}
\nonumber
\big| \omega_\lambda^i(x)- f_i (x) \big| =
\end{equation}
\begin{equation}
\nonumber
\big| \int_{\mathbb R^n}  \big( g_\lambda(y) |y|^{\frac{1}{\lambda^2}} - 1 \big) \frac{y_{i}}{|y|^{n}} D_{i}f(x-y)dy \big|
\end{equation}
\begin{equation}
\nonumber
\leq \int_{\mathbb R^n} \frac{\big|g_\lambda(y) |y|^{\frac{1}{\lambda^2}} - 1\big| }{|y|^{n-1}} \big| D_{i}f(x-y)\big| dy
\end{equation}
\begin{equation}
\nonumber
\leq \int_{\mathbb R^n} \frac{ \big|g_\lambda(y) |y|^{\frac{1}{\lambda^2}} - g_\lambda(y)\big| }{|y|^{n-1}} \big|D_{i}f(x-y)\big|dy + \int_{\mathbb R^n} \frac{ \big|g_\lambda(y) - 1 \big| }{|y|^{n-1}} \big|D_{i}f(x-y)\big|dy
\end{equation}
Since the support of $g_\lambda$ is in the ring $\Delta(\frac{1}{\lambda}, 2 \lambda)$, taking the $L^2$-norm for the first term and using the convolution theorem, we have
\begin{equation}
\nonumber
\big( \int_{\mathbb R^n} \big| \int_{\mathbb R^n} \frac{\big|g_\lambda(y) (|y|^{\frac{1}{\lambda^2}}-1) \big|}{|y|^{n-1}} | D_if(x-y)| dy  \big|^2dx \big)^\frac{1}{2}\leq
\end{equation}
\begin{eqnarray*}
&\leq &\big( \int_{\mathbb R^n} \big| \int_{\Delta (\frac{1}{\lambda},2\lambda)} \frac{||y|^{\frac{1}{\lambda^2}}-1|}{|y|^{n-1}} |D_if(x-y)|dy \big|^2dx \big)^\frac{1}{2}\\
&\leq &\|D_if\|_{L^2} \int_{\Delta (\frac{1}{\lambda},2\lambda)} \frac{||x|^\frac{1}{\lambda^2}-1|}{|x|^{n-1}}dx\,.
\end{eqnarray*}
Using polar coordinates we may see that the last integral converges to 0, as $\lambda \rightarrow \infty$. Also, for the second term, we have
\begin{equation}
\nonumber
\int_{\mathbb R^n}  \frac{|g_\lambda(y)-1|}{|y|^{n-1}} |D_if(x-y)|dy
\end{equation}
\begin{equation}
\nonumber
\leq \int_{B(0, \frac{2}{\lambda})} \frac{|D_if(x-y)|}{|y|^{n-1}}dy + \int_{\mathbb R^n \setminus B(0, \lambda)} \frac{|D_if(x-y)|}{|y|^{n-1}}dy
\end{equation}
But, by the convolution Theorem,
\begin{equation}
\nonumber
\int_{\mathbb R^{n}} \big | \int_{\mathbb R^{n}} \frac{\chi_{B(0,\frac{2}{\lambda})}(y) }{|y|^{n-1}} \big | D_{i}f(x-y) \big | dy  \big |dx  
\leq \| D_{i}f\|_{L^1} \int_{B(0,\frac{2}{\lambda})} \frac{dx}{|x|^{n-1}}\,.
\end{equation}
Also
\begin{equation}
\nonumber
\int_{\mathbb R^n \setminus B(0, \lambda)} \frac{|D_if(x-y)|}{|y|^{n-1}}dy \leq \lambda ^{1-n} \|D_{i}f\|_{L^{1}}.
\end{equation}
So, using a subsequence if necessary, the second term tends to zero a.e as well.

Hence
$$| \omega_\lambda^i(x)- f_i(x) | \rightarrow 0\,,$$
a.e., as $\lambda \rightarrow \infty$.

Using this, the fact that $|\omega_\lambda^i|$ is bounded by an $L^{p(\cdot)}$-function and the dominated convergence theorem for $L^{p(\cdot)}$-spaces, we have that
\begin{equation}
\label{eq4}
\big\| \omega_\lambda^i- f_i \big\|_{L^{p(\cdot)}} \rightarrow 0\,,\text{ as }\lambda \rightarrow \infty\,,
\end{equation}
for all $1\leq i\leq n$.

As we have said in the Preliminaries, $f$ can be written as
$$\frac{1}{n\omega_n}\sum_{i=1}^n f_i\,.$$
From this and (\ref{eq4}), we get that
\begin{equation}
\nonumber
\big\| \sum_{i=1}^n \frac{1}{n\omega_n} \omega_\lambda^i - f \big \|_{L^{p(\cdot)}} \rightarrow 0\,,\text{ as }\lambda \rightarrow \infty\,.
\end{equation}
Now, if we were able to prove that $D_j\big( \sum_{i=1}^n \frac{1}{n\omega_n} \omega_\lambda^i \big)$, for $1\leq i, j \leq n$, is bounded above by an $L^{p(\cdot)}$-function, then from the above convergence, it would follow that
\begin{equation}
\nonumber
\big\| D_j \big( \sum_{i=1}^n \frac{1}{n\omega_n} \omega_\lambda^i \big) - D_jf \big\|_{L^{p(\cdot)}} \rightarrow 0\,, \text{ as }\lambda \rightarrow \infty\,.
\end{equation}
To this end using Lemma \ref{lemma1}, we have that
\begin{equation}
\label{eq6}
\int_{\mathbb R^n} \frac{|D_if(x-y)|}{|y|^{n-\frac{1}{\lambda^2}}}dy \rightarrow |D_if(x)|\,, \text{ a.e in } \mathbb R^n\,.
\end{equation}
Lemma \ref{lemma2} then gives
\begin{equation}
\nonumber
|D_j\omega_\lambda^i(x)| =
\end{equation}
\begin{equation}
\nonumber
\big| \int_{\mathbb R^n} D_j\big(g_\lambda(y)\big) \frac{y_{i}}{|y|^{n-\frac{1}{\lambda^2}}} D_if(x-y)dy + \int_{\mathbb R^n} \big(g_\lambda(y)\big)D_j \big(\frac{y_{i}}{|y|^{n-\frac{1}{\lambda^2}}}\big)D_if(x-y)dy \big|
\end{equation}
\begin{eqnarray*}
&\leq & (n+3) \int_{\mathbb R^n} \frac{|D_if(x-y)|}{|y|^{n-\frac{1}{\lambda^2}}}dy\\
&\leq &(n+3) \int_{B(0,1)} \frac{|D_if(x-y)|}{|y|^{n-\frac{1}{\lambda^2}}}dy+ (n+3) \int_{\mathbb R^n \setminus B(0,1)} \frac{|D_if(x-y)|}{|y|^{n-1}}dy\,.
\end{eqnarray*}
The second term  is bounded by the Riesz potential of an $L^{p(\cdot)}(\mathbb R^n)$ function and hence  by hypothesis it is in $L^{p(\cdot)}(\mathbb R^n)$.  The sequence
$$\int_{B(0,1)} \frac{|D_if(x-y)|}{|y|^{n-\frac{1}{\lambda^2}}}dy$$
is increasing in $\lambda$ and bounded by the convolution
$$\frac{1}{|x|^{n-\frac{1}{\lambda^2}}} * D_if\,.$$
Thus by (\ref{eq6}) it is bounded by $|D_i f (x)| $, which is a function in $L^{p(\cdot)}(\mathbb R^n)$. Hence $|D_j \omega_\lambda^i|$ is bounded by an $L^{p(\cdot)}(\mathbb R^n)$-function. So
$$D_j \big( \sum_{i=1}^n \frac{1}{n\omega_n} \omega_\lambda^i \big)$$
is bounded by an $L^{p(\cdot)}(\mathbb R^n)$-function too.

Concluding, we have shown that
\begin{equation}
\nonumber
\big\|  \sum_{i=1}^n \frac{1}{n\omega_n} \omega_{\lambda} - f \big \|_{L^{p(\cdot)}} \rightarrow 0\,,\text{ as }\lambda \rightarrow \infty
\end{equation}
and
\begin{equation}
\nonumber
\big\| D_j \big(  \sum_{i=1}^n \frac{1}{n\omega_n} \omega_\lambda^i \big) - D_jf \big\|_{L^{p(\cdot)}} \rightarrow 0\,,\text{ as }\lambda \rightarrow \infty
\end{equation}
and these two facts lead us to the desired result.

\end{proof}

\section{Density results involving the dimension of the space}
Our first application of Theorem \ref{theorem1} shows that if the exponent dominates the dimension of the space, then we have density. Note that this also answers affirmatively a question posed by P. Hasto in \cite{hasto}, for $\Omega=\mathbb R^n$.
\begin{theorem}
If $p_-\geq n$, then
$C_0^\infty (\mathbb R^n)$ is dense in $W^{1,p(\cdot)} (\mathbb R^n)$.
\end{theorem}
\begin{proof}
Let $f\in L^{p(\cdot)}(\mathbb{R}^n)$ be a function with compact support. Since $p_-\geq 2$ we have that
$$f\in L^{\frac{2n}{n+2}}(\mathbb R^n)$$
and hence by what we have said in the Preliminaries $I_1(f)\in L^2(\mathbb R^n)$.

On the other hand since $p_-\geq n$ and $f$ has compact support, we have that $f\in L^q(\mathbb R^n)$, for any $q
\leq n$. Choosing $q<n$ such that $\frac{nq}{n-q}\geq p_+$ we have that
$I_1(f)\in L^\frac{nq}{n-q}(\mathbb R^n)$.

Combining the above we get that
$$I_1(f)\in L^2(\mathbb R^n)\cap L^\frac{nq}{n-q}(\mathbb R^n)\,,$$
which by Theorem \ref{variable} implies that $I_1(f)\in L^{p(\cdot)}(\mathbb R^n)$. Using Theorem \ref{theorem1} we have that
$C_0^\infty (\mathbb R^n)$ is dense in $W^{1,p(\cdot)} (\mathbb R^n)$.
\end{proof}
If $p_-< n$ one has  to impose some restriction on $p_+$ as well. In particular we have the following.
\begin{theorem}
If $2\leq p_-< n$ and $ p_+ \leq \frac{n p_-}{n- p_-}$, then
$C_0^\infty (\mathbb R^n)$ is dense in $W^{1,p(\cdot)} (\mathbb R^n)$.
\end{theorem}
\begin{proof}
Let $f\in L^{p(\cdot)}(\mathbb{R}^n)$ be a function with compact support. As before by the fact that $p_-\geq 2$ we have that $I_1(f)\in L^2(\mathbb R^n)$.

Also, by (\ref{riesz1})
$$I_1(f)\in L^\frac{np_-}{n-p_-}(\mathbb R^n)\,.$$
So
$$I_1(f)\in  L^2(\mathbb R^n)\cap L^\frac{np_-}{n-p_-}(\mathbb R^n)$$
and thus by Theorem \ref{variable} we have that $I_1(f)\in L^{p(\cdot)}(\mathbb R^n)$. The result follows again by Theorem \ref{theorem1}.
\end{proof}
\section{Density results involving the maximal operator}
Given $f\in L^1_{loc}(\mathbb R^n)$, then $Mf$, the Hardy-Littlewood maximal function of $f$, is defined for any $x\in\mathbb R^n$, by
\begin{equation}
\nonumber
Mf(x)=\sup_{Q\in x}\frac{1}{|Q|}\int_Q |f(y)|\,dy\,,
\end{equation}
where the supremum is taken over all $Q\subset\mathbb R^n$ that contain $x$ and whose sides are parallel to the coordinate axes, see \cite[Definition 3.1]{cruz1}.

Unlike the classical case the maximal operator is not necessarily bounded in $L^{p(\cdot)}({\mathbb R^n})$. For our purpose local boundedness is enough: in particular we say that the maximal operator is locally bounded in $L^{p(\cdot)}({\mathbb R^n})$ if for every  bounded open subset $\Omega$ of $\mathbb R^n$, the maximal operator is bounded in  $L^{p(\cdot)}(\Omega)$, see \cite[Definition 6.13]{cruz1}.
We have the following
\begin{theorem}
\label{maximal}
If $p_-\geq 2$ and the maximal operator is locally bounded in $L^{p(\cdot)}(\mathbb R^n)$, then $C_0^\infty (\mathbb R^n)$ is dense in $W^{1,p(\cdot)} (\mathbb R^n)$.
\end{theorem}
\begin{proof}
Let $f\in L^{p(\cdot)}({\mathbb R^n})$ with compact support. Using an argument similar to that of \cite[Proposition 6.22]{cruz1}, we have that
\begin{equation}
\nonumber
\|I_1(f)\|_{p(\cdot)}\leq c\|f\|_{p(\cdot)}\,,
\end{equation}
where the constant $c>0$ depends on the bound of the maximal operator on the support of $f$ and also on its diameter. Hence using Theorem \ref{theorem1} we get the required result.
\end{proof}
\begin{remark}
Note that as in \cite[Proposition 6.22]{cruz1} we could have assumed that the maximal operator is locally bounded in $L^{p'(\cdot)}(\mathbb R^n)$, where as in the classical case $\frac{1}{p(x)}+\frac{1}{p'(x)}=1$,  for all $x\in\mathbb R^n$.
\end{remark}
A sufficient condition for the local boundedness of the maximal operator is local log-H{\"o}lder continuity, see \cite[p. 246]{cruz1}. Hence by Theorem \ref{maximal} we have the following
\begin{theorem}
\label{log-Holder}
If $p_-\geq 2$ and $p$ is locally log-H{\"o}lder continuous, then $C_0^\infty (\mathbb R^n)$ is dense in $W^{1,p(\cdot)} (\mathbb R^n)$.
\end{theorem}
\begin{remark}
As we have already mentioned in the introduction our approach in order to assure the validity of Theorem \ref{log-Holder} provides us with an alternative proof of this important result, when $p_-\geq 2$.
\end{remark}


\begin{thebibliography}{00}
\bibitem{cruz} D. Cruz-Uribe, A. Fiorenza, \textit{Approximate identities in variable $L^p$ spaces}, Math. Nachr. \textbf{289} (2007), 256--270.
\bibitem{cruz1} D. Cruz-Uribe, A. Fiorenza, \textit{Variable Lebesgue Spaces}, Birkh\"auser/Springer, Heidelberg, 2013.
\bibitem{dhhr} L. Diening, P. Harjulehto, P. Hasto and M. Ruzicka, \textit{Lebesgue and Sobolev Spaces with Variable Exponents}, Springer, Heidelberg, 2011.
\bibitem{dhn} L. Diening, \textit{Maximal function on generalized Lebesgue spaces $L^{p(\cdot)}$}, Math. Inequal.
Appl. \textbf{7} (2004), 245--253.
\bibitem{edmunds} D. Edmunds, J. R{\'a}kosn{\'{\i}}k, \textit{Density of smooth functions in
 $W^{k,p(x)}(\Omega)$}, Proc. Roy. Soc. London Ser. A \textbf{437} (1992), 229--236.
\bibitem{fan} X. L. Fan, S. Wang, D. Zhao, \textit{Density of $C^\infty(\Omega)$ in $W^{1,p(x)}(\Omega)$ with discontinuous exponent $p(x)$}, Math. Nachr. \textbf{279} (2006), 142--149.
\bibitem{gilbarg} D. Gilbarg, N. S. Trudinger, \textit{Elliptic Partial Differential Equations of Second Order}, Reprint of the 1998 edition, Springer-Verlag, Berlin, 2001.
\bibitem{hasto} P. Hasto, \textit{On the density of continuous functions in variable exponent Sobolev space}, Rev. Mat. Iberoam. \textbf{23} (2007), 215--237.
\bibitem{ll} E. H. Lieb, M. Loss, \textit{Analysis}, Graduate Studies in Mathematics \textbf{14}, Second edition, American Mathematical Society, Providence, RI, 2001.
\bibitem{Riesz} M. Riesz, \textit{L'integrale de Riemann-Liouville et le probleme de Cauchy}, Acta Math. \textbf{81} (1948), 1--222.
\bibitem{samko1} S. Samko, \textit{Denseness of $C_0^\infty(\mathbb R^n)$ functions in the generalized Sobolev spaces $W^{m, p(x)}(\mathbb R^n)$}. In direct and inverse problems of mathematical physics, (Newark, DE, 1997), volume 5 of Int. Soc. Anal. Appl. Comput., Kluwer Acad. Publ., Dordrecht (2000), 1--10.
\bibitem{stein} {E. M. Stein}, \textit{{Singular integrals and differentiability properties of functions}}, Princeton University Press, Princeton, 1970.
\bibitem{zhikov2} V. V. Zhikov, \textit{Averaging of functionals of the calculus of variations and elasticity theory}, Math. USSR-Izv. \textbf{29} (1987), 33--65.
\bibitem{zhikov1} V. V. Zhikov, \textit{Density of smooth functions in Sobolev-Orlicz spaces}, J Math. Sc. \textbf{132} (2006), 285--294.
\end{thebibliography}
\end{document}